\documentclass[a4paper]{article}
\usepackage[english]{babel}
\usepackage[utf8]{inputenc}
\usepackage{amsmath}
\usepackage{graphicx}
\usepackage[colorinlistoftodos]{todonotes}

\usepackage{amsthm}
\usepackage{amsfonts}
\usepackage{amssymb}
\usepackage{amscd}
\usepackage{stmaryrd}
\usepackage{float}
\usepackage[all]{xy}
\usepackage{relsize}
\usepackage{geometry}

\usepackage{natbib}
\setcitestyle{authoryear,open={(},close={)}}

\usepackage{algpseudocode,algorithm,algorithmicx}

\algrenewcommand\algorithmicrequire{\textbf{Precondition:}}
\algrenewcommand\algorithmicensure{\textbf{Postcondition:}}

\algnewcommand\algorithmicforeach{\textbf{for each}}
\algdef{S}[FOR]{ForEach}[1]{\algorithmicforeach\ #1\ \algorithmicdo}

\newcommand{\A}{\ensuremath \mathcal{A}}
\newcommand{\B}{\ensuremath \mathcal{B}}

\newcommand{\Q}{\ensuremath \mathbb{Q}}
\newcommand{\R}{\ensuremath \mathbb{R}}
\newcommand{\C}{\ensuremath \mathbb{C}}
\newcommand{\Z}{\ensuremath \mathbb{Z}}
\newcommand{\N}{\ensuremath \mathbb{N}}
\newcommand{\isom}{\ensuremath \cong}
\newcommand{\gmax}{\ensuremath G_{W}^{\max}}
\newcommand{\gsmith}{\ensuremath G_{W}^{\text{Smith}}}

\theoremstyle{plain}
\newtheorem{theorem}{Theorem}[section]

\newtheorem{lemma}[theorem]{Lemma}

\theoremstyle{definition}
\newtheorem{definition}[theorem]{Definition}
\newtheorem{example}[theorem]{Example}

\newcommand{\thmref}[1]{Theorem \ref{#1}}

\newcommand{\lemref}[1]{Lemma \ref{#1}}

\newcommand{\exampleref}[1]{Example \ref{#1}}

\title{Two Algorithms to Compute Symmetry Groups for Landau-Ginzburg Models}

\author{Nathan Cordner\thanks{ncordner@bu.edu} \\ Boston University Department of Computer Science}

\date{June 27, 2018}

\begin{document}

\newgeometry{top = 1 in, left = 1 in, right = 1 in, bottom = 1 in}

\maketitle

\begin{abstract}
Landau-Ginzburg mirror symmetry studies isomorphisms between graded Frobenius algebras, known as A- and B-models. Fundamental to constructing these models is the computation of the finite, Abelian \emph{maximal symmetry group} $G_{W}^{\max}$ of a given polynomial $W$. For \emph{invertible} polynomials, which have the same number of monomials as variables, a generating set for this group can be computed efficiently by inverting the \emph{polynomial exponent matrix}. However, this method does not work for \emph{noninvertible} polynomials with more monomials than variables since the resulting exponent matrix is no longer square.

A previously conjectured algorithm to address this problem relies on intersecting groups generated from \emph{submatrices} of the exponent matrix. We prove that this method is correct, but intractable in general. We overcome intractability by presenting a group isomorphism based on the Smith normal form of the exponent matrix. We demonstrate an algorithm to compute $\gmax$ via this isomorphism, and show its efficiency in all cases.

\end{abstract}

\newpage 

\section{Introduction}

In the context of Landau-Ginzburg mirror symmetry, two different physical theories arise known as Landau-Ginzburg A- and B-models. These are graded Frobenius algebras, built using a nondegenerate weighted homogeneous polynomial $W$ and a related group of symmetries $G$ of $W$. \cite{FJR07} have constructed the A-model theories, which are a special case of what is known as \emph{FJRW theory}. \cite{IV}, \cite{Ka1,Ka2,Ka3}, and \cite{Kra09} have constructed the B-model theories, which correspond to an \emph{orbifolded Milnor ring}. In many cases, both A- and B-model theories extend to whole families of Frobenius algebras, called \emph{Frobeinus manifolds}.

For the so-called \emph{invertible} polynomials, \cite{BH}, \cite{Hen}, and \cite{Kra09} described the construction of a dual (or transpose) polynomial $W^{T}$ and a dual group $G^{T}$. One formulation of the Landau-Ginzburg mirror symmetry conjecture states that the A-model of a polynomial-group pair $(W,G)$ should be isomorphic to the B-model of the dual pair $(W^{T}, G^{T})$ on the level of graded Frobenius algebras, and is written as $\A[W,G] \isom \B[W^{T}, G^{T}]$.  \cite{Kra09}, \cite{FJJS11}, and \cite{FJR07} have proven this conjecture in many cases, although the proof of the full conjecture remains open.

One way to generalize mirror symmetry is to include a larger class of polynomials called \emph{noninvertible}. Together with the invertible polynomials, these make up the class of \emph{admissible} polynomials. Many of the mirror symmetry constructions immediately generalize for noninvertible polynomials, but not all \cite[see][]{Cor15}. An open area of research focuses on developing the ideas necessary for mirror symmetry to include all such admissible polynomials.

Fundamental to the construction of Landau-Ginzburg A- and B-models is the group of symmetries for a given polynomial $W$. The largest group allowed for an A-model using polynomial $W$ is the \emph{maximal symmetry group}, denoted $\gmax$. This is a finite, Abelian group when $W$ is admissible. In the case that $W$ is invertible, a generating set for $\gmax$ is given by the columns of the inverse polynomial exponent matrix. This method does not work when $W$ is noninvertible, since the shape of the exponent matrix is no longer square. 

In this paper we will focus on developing and analyzing two characterizations of $\gmax$ which yield computational methods that include the noninvertible case. In Section 3, we examine a previously conjectured characterization which relates $\gmax$ as the intersection of groups generated by all the submatrices of the exponent matrix of $W$. This is the natural extension of the method used to compute $\gmax$ for invertible polynomials (see \lemref{invert_cols_gmax}). In Section 4, we present an alternative characterization of $\gmax$ via a group isomorphism based on the Smith normal form of the exponent matrix.

In turn, these characterizations yield algorithms to compute $\gmax$. We will show that Algorithm \ref{alg:submatrix-gmax-improved} based off the submatrix characterization is correct, but not efficient. In particular, we exhibit a family of cases where its running time is exponential in the size of its input. We will also show that Algorithm \ref{alg:smith-form-gmax} based off the Smith normal form is correct, and that it runs in polynomial time in terms of the size of the exponent matrix.

\section{Background}

Here we will introduce some of the concepts needed to understand the theory of this paper. 

\subsection{Polynomials}

\begin{definition}
For a polynomial $W \in \C[x_{1},\dots,x_{n}]$, we say that $W$ is \textit{nondegenerate} if it has an isolated critical point at the origin.
\end{definition}

\begin{definition}
Let $W \in \C[x_{1},\dots,x_{n}]$. We say that $W$ is \textit{quasihomogeneous} if there exist positive rational numbers $q_{1},\dots,q_{n}$ such that for any $c \in \C$, $W(c^{q_{1}}x_{1},\dots,c^{q_{n}}x_{n}) = c W(x_{1},\dots,x_{n})$.
\end{definition}

We often refer to the $q_{i}$ as the \textit{quasihomogeneous weights} of a polynomial $W$, or just simply the \textit{weights} of $W$, and we write the weights in vector form $J = (q_{1}, \dots, q_{n})$. 

\begin{definition}\label{admissible}
$W \in \C[x_{1},\dots,x_{n}]$ is \emph{admissible} if $W$ is nondegenerate and quasihomogeneous with unique weights, having no monomials of the form $x_{i}x_{j}$ for $i \ne j$, $i, j \in \{1, \dots, n\}$.
\end{definition}

The condition that $W$ have no cross-term monomials is necessary for constructing the A-model \cite[see][]{FJR07}. Because the construction of $\A[W,G]$ requires an admissible polynomial, we will only be concerned with admissible polynomials in this paper. In order for a polynomial to be admissible, it needs to have at least as many monomials as variables. Otherwise its quasihomogeneous weights cannot be uniquely determined. We now state the main subdivision of the admissible polynomials.

\begin{definition}
Let $W$ be an admissible polynomial. We say that $W$ is \textit{invertible} if it has the same number of monomials as variables. If $W$ has more monomials than variables, then it is \textit{noninvertible}.
\end{definition}

We observe that if $W$ is invertible, we can rescale variables to force each coefficient $c_{i}$ to equal one---which we will do in this paper. The invertible polynomials can also be decomposed into sums of three types of polynomials, called the \emph{atomic types}.

\begin{theorem}[\cite{KS}]
Any invertible polynomial is the decoupled sum of polynomials in one of three atomic types:
\begin{align*}
\begin{array}{rl}\text{Fermat type: } & W = x^{a}, \\ 
\text{Loop type: } & W = x_{1}^{a_{1}}x_{2} + x_{2}^{a_{2}}x_{3} + \dots + x_{n}^{a_{n}}x_{1}, \\
\text{Chain type: } & W = x_{1}^{a_{1}}x_{2} + x_{2}^{a_{2}}x_{3} + \dots + x_{n}^{a_{n}}.
\end{array}
\end{align*}
\end{theorem}
We also assume that the $a_{i} \ge 2$ to avoid terms of the form $x_{i}x_{j}$ for $i \ne j$. This is the fundamental decomposition result for invertible polynomials.

\subsection{The Exponent Matrix}

For our computations, it is often useful to represent the polynomials in matrix form.

\begin{definition}
Let $W \in \C[x_{1},\dots, x_{n}]$. If we write $W$ as a sum of monomials $W = \sum_{i = 1}^{m} c_{i} \prod_{j = 1}^{n} x_{j}^{a_{ij}}$, then the associated \textit{exponent matrix} is defined to be $A = (a_{ij})$. 
\end{definition}

Note that $A$ is an $m \times n$ matrix, where $m$ is the number of monomials of $W$ and $n$ is the number of variables of $W$. Since the number of monomials must be greater than or equal to the number of variables for $W$ to be admissible, we will always have $m \ge n$. In the special case where $W$ is invertible, we get $m = n$. We note that these matrices are of full rank and are indeed invertible in this case, as the name of this class of polynomials suggests.

We would like to get a sense of how large $A$ can be in terms of $n$ and the quasihomogeneous weights of $W$. We define a norm for $A$ to pick out its largest entry, and write $\|A\| = \max_{i,j} |a_{ij}| = \max_{i,j} a_{ij}$ since the entries of $A$ are nonnegative integers.

\begin{lemma}
Given admissible $W$ with weight vector $\mathbf{q} = (q_{1}, \dots, q_{n})$, then $\|A\| \le \lfloor \max \{ \frac{1}{q_{i}} \} \rfloor$.
\end{lemma}

\begin{proof}
By the quasihomogeneous property, a monomial $x_{1}^{a_{1}} \dots x_{n}^{a_{n}}$ is admissible under $\mathbf{q}$ if and only if $\sum_{i = 1}^{n} a_{i} q_{i} = 1$. We produce the largest exponent by considering $q_{k} = \min\{q_{i}\}$, with $q_{k} = \frac{1}{a}$ for some $a \ge 2$. In this case, the monomial $x_{i}^{a}$ is admissible. Therefore, for an admissible polynomial $W$, $\|A\| \le \lfloor \max \{ \frac{1}{q_{i}} \} \rfloor$ and this is a tight upper bound.
\end{proof}

Also of interest is the number of rows that the exponent matrix allows for.

\begin{lemma}
There exist weight systems that admit exponentially many monomials.
\end{lemma}

\begin{proof}
Consider the homogeneous weight system $\mathbf{q} = (\frac{1}{a}, \dots, \frac{1}{a})$ for some $a \ge 2$. Each admissible monomial $x_{1}^{c_{1}}\dots x_{n}^{c_{n}}$ satisfies $\sum_{i = 1}^{n} \frac{1}{a} \cdot c_{i} = 1$ which implies that $\sum_{i = 1}^{n} c_{i} = a$. Therefore any solution to the equation $c_{1} + \dots + c_{n} = a$ in nonnegative integers yields a new monomial. We get the exact solution
\[m = \left(\begin{matrix} a + n - 1 \\ a \end{matrix} \right).\]

Weight systems of this form admit many monomials in terms of the number of variables $n$. Choosing $a = n + 1$, for example, we get
\[m = \left(\begin{matrix} 2n \\ n + 1 \end{matrix} \right) = \frac{(2n)!}{(n+1)!(n-1)!} = \frac{2n(2n-1)\cdots(n+2)}{(n-1)(n-2)\cdots(2)(1)} \ge 2 \cdots 2 \; (n-1 \text{ times})= 2^{n-1} .\]
Therefore it is possible for $m$ to be exponential in terms of $n$.

\end{proof}

Though $A$ can be exponentially large as an input in terms of $n$, we will focus on more reasonably sized problems. In the algorithms that follow, we will consider $A$ to be of polynomial size in $n$ if both $m$ and $\|A\|$ are of polynomial size in $n$.

\subsection{Symmetry Groups}

We are now ready to define the maximal symmetry group for admissible polynomials.

\begin{definition}\label{gmax}
Let $W$ be an admissible polynomial. We define the \textit{maximal diagonal symmetry group} of $W$ to be $G_{W}^{\max} = \{(\zeta_{1}, \dots, \zeta_{n}) \in (\C^\times)^{n} \mid W(\zeta_{1}x_{1},\dots,\zeta_{n}x_{n}) = W(x_{1},\dots,x_{n})    \}$.
\end{definition}

\cite{FJR07} and \cite{ABS11} show that $G_{W}^{\max}$ is finite and that each coordinate of every group element is a root of unity. The group operation in $G_{W}^{\max}$ is coordinate-wise multiplication. Converting to additive notation, we use the map $(e^{2\pi i \theta_{1}},\dots,e^{2\pi i \theta_{n}}) \mapsto (\theta_{1}, \dots, \theta_{n}) \mod \Z$ taking $G_{W}^{\max}$ to $(\Q / \Z)^{n}$. Hence we will write $G_{W}^{\max} = \{ \mathbf{g} \in (\Q/\Z)^{n} \mid A\mathbf{g} \in \Z^{m} \}$, where $A$ is the $m \times n$ exponent matrix of $W$. In this notation we have the following

\begin{definition}
The group $G_{W}^{\max}$ is a subgroup of $(\Q/\Z)^{n}$ with respect to coordinate-wise addition. For $\mathbf{g} \in G_{W}^{\max}$, we can write $\mathbf{g}$ uniquely as $(g_{1}, \dots, g_{n})$, where each $g_{i}$ is a rational number in the interval [0,1). The $g_{i}$ are called the \textit{phases} of $\mathbf{g}$.
\end{definition}

We will occasionally say that a group element represented by its phases is in its \emph{canonical form}. However, as a matter of convenience we will sometimes use equivalent representatives of the $g_{i}$ that lie outside the interval [0,1) to write down group elements. One such example of this occurs in the following

\begin{lemma}[see \cite{ABS11} or \cite{Kra09}]\label{invert_cols_gmax}
If $W$ is invertible, then $\gmax$ is generated by the columns of $A^{-1}$. We write $\gmax = \langle \text{cols of } A^{-1} \rangle$.
\end{lemma}

Though this method uses column vectors, throughout this paper we will usually represent group elements as row vectors. We also note that the columns of $A^{-1}$ are not necessarily in canonical form. However, the resulting group elements can certainly be represented by their phases. Note also that this characterization of $\gmax$ works only for invertible polynomials. When $W$ is noninvertible, its exponent matrix is no longer square and hence is not invertible. Generalizing the result of \lemref{invert_cols_gmax} to include noninvertible polynomials is the focus of the next section. One additional result that we will also later generalize (see \lemref{new_size_gmax}) is
\begin{lemma}[see \cite{ABS11}]\label{size_gmax}
If $W$ is invertible, then the size of $\gmax$ is equal to the absolute value of the determinant of the exponent matrix of $W$. We write $|\gmax| = |\det(A)|$. 
\end{lemma}

\section{Characterization of $\gmax$ via Submatrices}

In order to state our new formulation of the maximal symmetry group, we first need a notion of submatrix for a given matrix $A$.

\begin{definition}
Given an admissible polynomial $W$ with $m \times n$ exponent matrix $A$. Any matrix $A_{i}$ comprised of $n$ rows taken from $A$ is called a \emph{submatrix} of $A$. We say that two submatrices $A_{i}$, $A_{j}$ of $A$ are \emph{equivalent} if they differ only by a permutation of rows.
\end{definition}

The following characterization of $\gmax$, attributed to Drew Johnson \cite[]{Webb}, is one way to compute $\gmax$. In the case where $m = n$, the characterization matches the method of computing $\gmax$ for an invertible polynomial, and thus extends the result of \lemref{invert_cols_gmax}.

\begin{theorem}\label{gmax_submatrix_thm}
Let $W$ be an admissible polynomial, and let the $A_{i}$ denote the invertible submatrices of the polynomial exponent matrix $A$ (up to equivalence). Then $\gmax = \bigcap_{A_{i}} \langle$cols of $A_{i}^{-1} \rangle$.
\end{theorem}

\begin{proof}
To avoid the trivial case, assume that $W$ is noninvertible. Let $W_{i}$ be a polynomial corresponding to the exponent matrix $A_{i}$ for all values of $i$. We will denote by $G_{W_{i}}^{\max}$ the group $\langle$cols of $A_{i}^{-1} \rangle$. Though some of the polynomials $W_{i}$ may not necessarily be nondegenerate, this notation still makes sense because the characterization of $\gmax$ in \lemref{invert_cols_gmax} depends only on $W$ being quasihomogeneous and having an invertible exponent matrix. 

We note that extra monomials in a polynomial $W$ become extra rows in an exponent matrix, and serve as further restrictions for elements allowed in $\gmax$. In our case, since $W$ = $W_{i}$ plus additional monomials, we have that $\gmax \subseteq  G_{W_{i}}^{\max}$ for each $i$. Therefore $\bigcap_{A_{i}} \langle$cols of $A_{i}^{-1} \rangle = \bigcap_{W_{i}} G_{W_{i}}^{\max} \supseteq \gmax$.

Now choose arbitrary nonequivalent submatrices $A_{i}$ and $A_{j}$. Let $W_{i}$ and $W_{j}$ be corresponding polynomials for $A_{i}$ and $A_{j}$. Since $A_{i}$ and $A_{j}$ are invertible $n \times n$ matrices, their corresponding polynomials will be polynomials of $n$ variables. We notice that $\langle$cols of $A_{i}^{-1} \rangle \cap \langle$cols of $A_{j}^{-1} \rangle = G_{W_{i}}^{\max} \cap G_{W_{j}}^{\max}$. 

Now we have the following:
\begin{align*}
G_{W_{i}}^{\max} \cap G_{W_{j}}^{\max} &= \{\mathbf{g} \in (\Q / \Z)^{n} \mid A_{i} \mathbf{g} \in \Z^{n} \} \cap \{\mathbf{g} \in (\Q / \Z)^{n} \mid A_{j}\mathbf{g} \in \Z^{n} \}\\
&=  \{\mathbf{g} \in (\Q / \Z)^{n} \mid A_{i}\mathbf{g} \in \Z^{n}\text{ and }A_{j}\mathbf{g} \in \Z^{n} \}\\
&\subseteq \{\mathbf{g} \in (\Q / \Z)^{n} \mid A_{i,j}' \mathbf{g} \in \Z^{m'} \}\\
&\text{  where } A_{i,j}' \text{ is the } m' \times n \text{ matrix containing all rows of } A_{i}, A_{j} \text{ (removing duplicates)}\\
&= G_{W_{i} + W_{j}}^{\max}.
\end{align*}
In general we see that $G_{W_{1}}^{\max} \cap G_{W_{2}}^{\max} \cap \dots \cap G_{W_{r}}^{\max} \subseteq G_{W_{1} + W_{2} + \dots + W_{r}}^{\max}$, which yields
\begin{align*}
\bigcap_{A_{i}} \langle \text{cols of }A_{i}^{-1} \rangle = \bigcap_{W_{i}} G_{W_{i}}^{\max} \subseteq G_{W_{1} + W_{2} + \dots + W_{r}}^{\max}.
\end{align*}

Because the exponent matrix $A$ has rank $n$, there are $n$ rows of $A$ that form a basis of $\R^{n}$. Let $B_{n} = \{\mathbf{a}_{1},\dots,\mathbf{a}_{n} \}$ be a basis of $\R^{n}$ comprised of $n$ row vectors from $A$. Let $\mathbf{a}_{n+1}, \dots, \mathbf{a}_{m}$ be the remaining rows of $A$. The sets $\{\mathbf{a}_{n+1} \}, \dots, \{\mathbf{a}_{m} \}$ are linearly independent sets in $\R^{n}$. By the Replacement Theorem of linear algebra, we can create new bases of $\R^{n}$, say $B_{n+1}, \dots, B_{m}$, such that $\mathbf{a}_{n+1} \in B_{n+1}, \dots, \mathbf{a}_{m} \in B_{m}$. Each row of $A$ can be represented in a basis of $\R^{n}$. 

Let $A_{n}, A_{n+1}, \dots, A_{m}$ be matrices with rows made up from the basis vectors in $B_{n}, B_{n+1}, \dots, B_{m}$, respectively. Then $A_{n}, A_{n+1}, \dots, A_{m}$ are all $n \times n$ invertible matrices, with rows originally taken from $A$. Let $W_{n}, W_{n+1}, \dots, W_{m}$ be the corresponding polynomials to these exponent matrices. We see that the polynomials $W$ and $W_{n} + W_{n+1} + \dots + W_{m}$ have all the same monomials. Therefore the polynomials $W$ and $W_{1} + W_{2} + \dots + W_{r}$ defined above also share all the same monomials. Since $\gmax$ is independent from the coefficients of $W$, we have that $G_{W_{1} + W_{2} + \dots + W_{r}}^{\max} = \gmax$. Therefore we see that $\bigcap_{A_{i}} \langle \text{cols of }A_{i}^{-1} \rangle \subseteq \gmax$. 

\end{proof}

We now have the following preliminary algorithm to compute $\gmax$ even when $W$ is noninvertible. The algorithm accepts the polynomial exponent matrix $A$ as input, and returns the entire set of group elements in canonical form that comprise $\gmax$.

\setcounter{algorithm}{-1}
\begin{algorithm}[H]
  \caption{Compute the generators of $\gmax$ via submatrices
    \label{alg:submatrix-gmax}}
  \begin{algorithmic}[1]
    \Statex
    \Function{GmaxSubmatrix}{$A$} \Comment{$A$ is the $m \times n$ exponent matrix}
    \ForEach{submatrix $A_{i}$ of $A$} \Comment{Up to equivalence}
      \If{$A_{i}$ is invertible}
      \State Compute the inverse $A_{i}^{-1}$ of $A_{i}$
      \State Generate the group $G_{i} \le (\Q / \Z)^{n}$ from the columns of $A_{i}^{-1}$ 
      \EndIf
      \EndFor
    \State \Return{the intersection of groups $\bigcap_{i} G_{i}$} \Comment{Use canonical representatives for group elements}
    \EndFunction
  \end{algorithmic}
\end{algorithm}

Before we perform our analysis, we will see how this algorithm works on a small example.

\begin{example}\label{ex-alg-1}
Consider the polynomial $W = x^{3} + y^{3} + x^{2}y$. One can quickly verify that $W$ is quasihomogeneous with weights $\left(\frac{1}{3}, \frac{1}{3} \right)$, and has its gradient vanishing only at the origin. Thus $W$ is admissible. We can represent its exponent matrix as
\begin{align*}
A = \begin{bmatrix}3 & 0 \\ 0 & 3 \\ 2 & 1 \end{bmatrix}.
\end{align*}
We have three different ways to choose two rows from $A$, so we form the submatrices
\begin{align*}
A_{1} = \begin{bmatrix}3 & 0 \\ 0 & 3 \end{bmatrix}, \; A_{2} = \begin{bmatrix}3 & 0 \\ 2 & 1 \end{bmatrix}, \; A_{3} = \begin{bmatrix} 0 & 3 \\ 2 & 1 \end{bmatrix}.
\end{align*}
Taking inverses of each matrix, we get 
\begin{align*}
A_{1}^{-1} = \begin{bmatrix}1/3 & 0 \\ 0 & 1/3 \end{bmatrix}, \; A_{2}^{-1} = \begin{bmatrix}1/3 & 0 \\ -2/3 & 1 \end{bmatrix}, \; A_{3} = \begin{bmatrix} -1/6 & 1/2 \\ 1/3 & 0 \end{bmatrix}.
\end{align*}
From the columns of $A_{1}^{-1}$ (writing column vectors as row vectors) we get the group generators $\left(\frac{1}{3}, 0 \right)$ and $\left(0, \frac{1}{3}\right)$. These produce the nine-element set
\[G_{1} = \left\{(0,0), (1/3, 0), (2/3, 0), (0,1/3),(1/3,1/3),(2/3,1/3),(0,2/3),(1/3,2/3),(2/3,2/3) \right\}. \]
Putting the columns of $A_{2}^{-1}$ in canonical form, we get the single generator $\left(\frac{1}{3}, \frac{1}{3} \right)$. The second column becomes (0,0) in canonical form, which contributes nothing. We get the three-element set
\[G_{2} = \{(0,0), (1/3,1/3), (2/3,2/3)\}.\]
Finally, from the columns of $A_{3}^{-1}$, we get generators $\left(\frac{5}{6}, \frac{1}{3} \right)$ and $\left(\frac{1}{2}, 0 \right)$. This yields a six-element set
\[G_{3} = \{(0,0), (1/2,0), (5/6,1/3), (1/3,1/3), (2/3,2/3), (1/6, 2/3)\}. \]
Taking the intersection of $G_{1}$, $G_{2}$, and $G_{3}$ yields our symmetry group
\[G_{W}^{\max} = \{(0,0), (1/3,1/3), (2/3,2/3)\}.\]
\end{example}

From our example, we note that there are a few slight improvements we can make to Algorithm \ref{alg:submatrix-gmax}. First, we can keep a running list of candidate group elements by computing the intersections of the $G_{i}$ inside the \texttt{for} loop. Second, we see that $\gmax = G_{2}$. In hindsight, computing $G_{3}$ was a little unnecessary since we ultimately got $G_{2} \subsetneq G_{3}$.

It turns out that the smallest group $\gmax$ can be is the one generated by the vector of quasihomogeneous weights. This follows by noting that $\mathbf{q} \in \gmax$ by the quasihomogeneous condition since $A \mathbf{q} = \mathbf{1}_{m} \in \Z^{m}$. Notice in \exampleref{ex-alg-1} that $\gmax$ is the group generated by the weights $\left( \frac{1}{3}, \frac{1}{3} \right)$, so this is a tight lower bound on the size of $\gmax$. If at any point in our computation our intersection of groups so far has size equal to the size of the group generated by $\mathbf{q}$, we can exit the loop and output our group elements. This gives the slightly better

\begin{algorithm}[H]
  \caption{Compute the generators of $\gmax$ via submatrices
    \label{alg:submatrix-gmax-improved}}
  \begin{algorithmic}[1]
    \Statex
    \Function{GmaxSubmatrix}{$A$} \Comment{$A$ is the $m \times n$ exponent matrix}
    \State Set $H := (\Q / \Z)^{n}$
    \ForEach{submatrix $A_{i}$ of $A$} \Comment{Up to equivalence}
      \If{$A_{i}$ is invertible}
      \State Compute the inverse $A_{i}^{-1}$ of $A_{i}$
      \State Generate the group $G_{i} \le (\Q / \Z)^{n}$ from the columns of $A_{i}^{-1}$ 
      \State Set $H := H \cap G_{i}$ \Comment{Use canonical representatives for group elements}
      \If{$|H| = |\langle \mathbf{q} \rangle |$}
          \State \Return{H}
      \EndIf
      \EndIf
      \EndFor
    \State \Return{H} 
    \EndFunction
  \end{algorithmic}
\end{algorithm}

Now for the analysis. By \thmref{gmax_submatrix_thm}, Algorithm \ref{alg:submatrix-gmax-improved} is correct---that is, it computes $\gmax$. However, it can be intractable even on polynomially sized input. We will consider a class of bad examples.

\begin{lemma}\label{example_poly}
For all even $n \ge 4$, $W_{n} = x_{1}^{2n} + \dots + x_{n}^{2n} + x_{1}^{n}x_{2}^{n} + \dots + x_{n}^{n}x_{1}^{n}$ is admissible.
\end{lemma}

\begin{proof}
Here $W_{n}$ is the sum of $n$ Fermat monomials with exponent $2n$, followed by $n$ extra monomials. Note that $W_{n}$ has weight system $\mathbf{q} = (\frac{1}{2n}, \dots, \frac{1}{2n})$, since any monomial of the form $x_{i}^{n} x_{j}^{n}$ (possibly with $i = j$) satisfies $n \cdot \frac{1}{2n} + n \cdot \frac{1}{2n} = 1$. To show that $W_{n}$ is nondegenerate, we compute
\begin{align*}
\nabla W_{n} &= \langle 2n x_{1}^{2n - 1} + nx_{1}^{2n-1}x_{2}^{n} + nx_{n}^{n}x_{1}^{n-1}, \dots, 2n x_{n}^{2n - 1} + nx_{n}^{2n-1}x_{1}^{n} + nx_{n-1}^{n}x_{n}^{n-1}   \rangle \\
&= \langle nx_{1}^{n-1}(2x_{1}^{n} + x_{2}^{n} + x_{n}^{n}), \dots,n x_{n}^{n-1}(2x_{n}^{n} + x_{1}^{n} + x_{n-1}^{n})  \rangle.
\end{align*}
$W_{n}$ has a critical point when $x_{1} = \dots = x_{n} = 0$. Notice that if $x_{i} = 0$, then $(\nabla W_{n})_{i} = 0$. And if $x_{i} \ne 0$, then in order for the $i$th coordinate of the gradient to vanish we require the $2x_{i}^{n} + x_{j}^{n} + x_{k}^{n}$ term to vanish. But since $n$ is even, each variable raised to the $n$th power is nonnegative. And since we assumed $x_{i} \ne 0$, then $2x_{i}^{n} + x_{j}^{n} + x_{k}^{n}$ has to be positive. Therefore $\nabla W_{n}$ only vanishes at the origin. Hence $W_{n}$ is nondegenerate.
\end{proof}

Notice that for $W_{n}$ the largest entry of $A$ is $2n$ (thus $\|A\| = 2n$), and the number of monomials of $W_{n}$ is $m = 2n$. Therefore the input to Algorithm \ref{alg:submatrix-gmax-improved} is of polynomial size to compute $G_{W_{n}}^{\max}$. However, the computation time quickly gets large.

\begin{theorem}\label{submatrix_alg_analysis}
Algorithm \ref{alg:submatrix-gmax-improved} requires exponential time to compute $G_{W_{n}}^{\max}$ for all even $n \ge 4$. 
\end{theorem}

\begin{proof}
Here the group generated by the weights vector is
\[\langle \mathbf{q} \rangle = \left\langle \left(1/2n, \dots, 1/2n \right)\right\rangle. \]
If $G_{W_{n}}^{\max} = \langle \mathbf{q} \rangle$, then every group element in $G_{W_{n}}^{\max}$ would have the canonical form $\left(\frac{c}{2n}, \dots, \frac{c}{2n} \right)$ for some $c \in \{0, \dots, 2n -1\}$. But we notice that vectors of the form $\left(0, \dots, \frac{1}{n}, \dots, 0\right) \in G_{W_{n}}^{\max}$ since both $2n \cdot \frac{1}{n}$ (coming from the monomials $x_{i}^{2n}$) and $n \cdot \frac{1}{n}$ (coming from the monomials $x_{i}^{n} x_{j}^{n}$) are integers. This shows that $\langle \mathbf{q} \rangle \lneq G_{W_{n}}^{\max}$, which means that Algorithm \ref{alg:submatrix-gmax-improved} does not exit early.

Hence the \texttt{for} loop on line 3 of the algorithm iterates over all the $\left( \begin{matrix} m \\ n \end{matrix} \right)$ submatrices of $A$. Plugging in $m = 2n$, we see that as $n$ gets large the inequality
\begin{align*}
\left( \begin{matrix} 2n \\ n \end{matrix} \right) \ge \frac{2^{2n}}{2n + 1} 
\end{align*}
holds via Stirling's formula. So Algorithm \ref{alg:submatrix-gmax-improved} takes at least $\frac{2^{2n}}{2n + 1}$ steps, which is exponential in the size of the input matrix $A$.

\end{proof}

\section{Characterization of $\gmax$ via the Smith Normal Form}

An alternative characterization for $\gmax$, which does not extend the result of \lemref{invert_cols_gmax}, is obtained through the Smith normal form of $A$.

\begin{definition}
The \emph{Smith normal form} of an $m \times n$ matrix $M$ over $\Z$ of rank $n$ is a matrix $S$ over $\Z$ that satisfies the following properties: there exist matrices $P$ and $Q$ such that $S = P M Q$, where $P$ is an invertible $m \times m$ matrix over $\Z$, $Q$ is an invertible $n \times n$ matrix over $\Z$, and $S$ is an $m \times n$ matrix of the form 
\begin{align*}
\begin{bmatrix} 
a_{1} & 0 & \dots & 0 \\
0 & a_{2} & \dots & 0 \\
\vdots & \vdots & \ddots & \vdots \\
0 & 0 & \dots & a_{n} \\
0 & 0 & \dots & 0 \\
\vdots & \vdots & \ddots & \vdots \\
0 & 0 & \dots & 0
\end{bmatrix}.
\end{align*}
The $a_{i}$ are sometimes referred to as \emph{invariant factors}. Here each $a_{i} \in \N \setminus \{0\}$, and for all $i \in \{2, \dots, n \}$ we have $a_{i-1} \mid a_{i}$. Because $M$ has rank $n$, none of the $a_{i} = 0$.
\end{definition}

Given an admissible polynomial $W$, we know that its $m \times n$ exponent matrix $A$ has full rank with integer-valued entries. Therefore, we can have the following

\begin{definition}
Define a new set $\gsmith$ as 
\begin{align*}
\gsmith = \{\mathbf{g} \in (\Q/\Z)^{n} \mid S\mathbf{g} \in \Z^{m} \},
\end{align*}
where $S$ is the Smith normal form of the exponent matrix of $W$. 
\end{definition}

We immediately see that $\gsmith$ is itself a group, and is a subgroup of $(\Q / \Z)^{n}$.

\begin{lemma}
$G_{W}^{\text{Smith}}$ is a subgroup of $(\Q / \Z)^{n}$. 
\end{lemma}

\begin{proof}
$\mathbf{0}_{n} \in \gsmith$, since $S\mathbf{0}_{n} = \mathbf{0}_{m} \in \Z^{m}$. Therefore $\gsmith$ is not empty. Now for any $\mathbf{x}, \mathbf{y} \in \gsmith$,
\begin{align*}
S(\mathbf{x} - \mathbf{y}) = S\mathbf{x} - S\mathbf{y} &= \mathbf{b}_{1} - \mathbf{b}_{2}, \text{ where } \mathbf{b}_{1},\mathbf{b}_{2} \in \Z^{m},\\
&= \mathbf{b} \in \Z^{m}.
\end{align*}
Therefore, by the Subgroup Criterion, $\gsmith \le (\Q/\Z)^{n}$. 
\end{proof}

We also observe that $\gsmith$ is isomorphic to our original group $\gmax$. 

\begin{theorem}\label{gmax_isom_thm}
$\gsmith\isom \gmax$ as finite Abelian groups.
\end{theorem}

\begin{proof}
Let $S = PAQ$ be the Smith normal form of $A$. We then have $A = P^{-1} S Q^{-1}$. Define a map $\phi : \gsmith \rightarrow \gmax$ by the rule $\mathbf{x} \mapsto Q\mathbf{x}$. Notice that for any $\mathbf{x} \in \gsmith$,
\begin{align*}
A\phi(\mathbf{x}) = AQ\mathbf{x} = (P^{-1}SQ^{-1})Q\mathbf{x} = P^{-1}S\mathbf{x}.
\end{align*}
Since $\mathbf{x} \in \gsmith$ then $S\mathbf{x} \in \Z^{m}$, which implies $P^{-1}S\mathbf{x} \in \Z^{m}$. This shows that $\phi(\mathbf{x}) \in \gmax$, so $\phi$ does map elements of $\gsmith$ into $\gmax$. 

For any $\mathbf{x}_{1}, \mathbf{x}_{2} \in \gsmith$, we write $\mathbf{x}_{1} \equiv \mathbf{x}_{2}$ if and only if there is some integer vector $\mathbf{b} \in \Z^{n}$ such that $\mathbf{x}_{1} = \mathbf{x}_{2} + \mathbf{b}$. Thus
\begin{align*}
 \mathbf{x}_{1} \equiv \mathbf{x}_{2} &\text{ if and only if } \mathbf{x}_{1} = \mathbf{x}_{2} + \mathbf{b}, \\
 &\text{ if and only if } Q \mathbf{x}_{1} = Q \mathbf{x}_{2} + Q\mathbf{b}, \\
 &\text{ if and only if } \phi(\mathbf{x}_{1}) = \phi(\mathbf{x}_{2}) + Q\mathbf{b}.
\end{align*}
Because $Q\mathbf{b} \in \Z^{n}$, we have that $ \phi(\mathbf{x}_{1}) \equiv \phi(\mathbf{x}_{2})$. Hence $\phi$ is well-defined. 


To show that $\phi$ is a bijection, we define an inverse map $\phi^{-1} : \gmax \rightarrow \gsmith$ by the rule $\mathbf{x} \mapsto Q^{-1}\mathbf{x}$. Then
\begin{align*}
\phi(\phi^{-1}(\mathbf{x})) &= \phi(Q^{-1}\mathbf{x}) = QQ^{-1}\mathbf{x} = \mathbf{x}, \\
\phi^{-1}(\phi(\mathbf{x})) &= \phi^{-1}(Q\mathbf{x}) = Q^{-1}Q\mathbf{x} = \mathbf{x}.
\end{align*}
We see that for any $\mathbf{x} \in \gmax$,
\begin{align*}
 S\phi^{-1}(\mathbf{x}) &= SQ^{-1}\mathbf{x} = (PAQ)Q^{-1}\mathbf{x} = PA\mathbf{x}.
\end{align*} 
Since $\mathbf{x} \in \gmax$ then $A\mathbf{x} \in \Z^{m}$, which implies $PA\mathbf{x} \in \Z^{m}$. This shows that $\phi^{-1}(\mathbf{x}) \in \gmax$, so $\phi^{-1}$ does map elements of $\gmax$ into $\gsmith$.

Also, for any $\mathbf{x}_{1}, \mathbf{x}_{2} \in \gmax$, we write $\mathbf{x}_{1} \equiv \mathbf{x}_{2}$ if and only if there is some integer vector $\mathbf{b} \in \Z^{n}$ such that $\mathbf{x}_{1} = \mathbf{x}_{2} + \mathbf{b}$. Thus,
\begin{align*}
 \mathbf{x}_{1} \equiv \mathbf{x}_{2} &\text{ if and only if } \mathbf{x}_{1} = \mathbf{x}_{2} + \mathbf{b}, \\
 &\text{ if and only if } Q^{-1} \mathbf{x}_{1} = Q^{-1} \mathbf{x}_{2} + Q^{-1}\mathbf{b}, \\
 &\text{ if and only if } \phi^{-1}(\mathbf{x}_{1}) = \phi^{-1}(\mathbf{x}_{2}) + Q^{-1}\mathbf{b}.
\end{align*}
Because $Q^{-1}\mathbf{b} \in \Z^{n}$, we have that $ \phi^{-1}(\mathbf{x}_{1}) \equiv \phi^{-1}(\mathbf{x}_{2})$. Hence $\phi^{-1}$ is well-defined. Since $\phi$ has a well-defined inverse map, it is a bijection.

Finally, to show that $\phi$ is a homomorphism, let $\mathbf{x}_{1}, \mathbf{x}_{2} \in \gsmith$. Then,
\begin{align*}
\phi(\mathbf{x}_{1} + \mathbf{x}_{2}) = Q(\mathbf{x}_{1} + \mathbf{x}_{2}) =  Q\mathbf{x}_{1} + Q\mathbf{x}_{2} = \phi(\mathbf{x}_{1}) + \phi(\mathbf{x}_{2}).
\end{align*}
Thus $\phi$ is an isomorphism, and we conclude that $\gsmith \isom \gmax$.
\end{proof}

We can now write down a generating set for $\gsmith$.

\begin{lemma}
$\gsmith = \langle (\frac{1}{a_{1}}, 0, \dots, 0), (0, \frac{1}{a_{2}}, 0, \dots, 0), \dots, (0, \dots, 0, \frac{1}{a_{n}}) \rangle$ .
\end{lemma}

\begin{proof}
Consider the $i$th generator $\mathbf{x}_{i} = (0, \dots, \frac{1}{a_{i}}, \dots, 0)$. Then 
\begin{align*}
S\mathbf{x}_{i} = (0, \dots, 1, \dots, 0, \dots, 0) \in \Z^{m}.
\end{align*}
Any integer linear combination of generators also produces an integer vector, because
\begin{align*}
S(c_{1}\mathbf{x}_{1} + \dots + c_{n}\mathbf{x}_{n}) &= c_{1}S\mathbf{x}_{1} + \dots + c_{n}S\mathbf{x}_{n} \\
&= (c_{1}, \dots, c_{n}, 0, \dots, 0) \in \Z^{m}.
\end{align*}
Thus any $\mathbf{x} \in \gsmith = (\frac{c_{1}}{a_{1}}, \dots, \frac{c_{n}}{a_{n}}) = c_{1}\mathbf{x}_{1} + \dots + c_{n}\mathbf{x}_{n}$.
\end{proof}

It may be of independent interest to note that $\gsmith$ immediately decomposes into the canonical form
\begin{align*}
\gsmith \isom \Z / a_{1} \Z \times \dots \times \Z / a_{n} \Z
\end{align*}
guaranteed by the Fundamental Theorem of Finitely Generated Abelian Groups. By the isomorphism of \thmref{gmax_isom_thm} we immediately get the same decomposition of invariant factors for $\gmax$. We can now generalize \lemref{size_gmax} with
\begin{lemma}\label{new_size_gmax}
Let $W$ be an admissible polynomial. If $D = \text{diag}[a_{i}]_{i = 1}^{n}$ (the first $n$ rows of the Smith normal form of the exponent matrix of $W$), then $|\gmax| = \det(D) =  \prod_{i = 1}^{n} a_{i}$.
\end{lemma}

Considering the isomorphism from $\gsmith$ to $\gmax$, we can now write (using column vectors)
\begin{align*}
\gmax = \langle Q(1/a_{1}, 0, \dots, 0)^{T}, Q(0, 1/a_{2}, 0, \dots, 0)^{T}, \dots, Q(0, \dots, 0, 1/a_{n})^{T} \rangle .
\end{align*}
We see that $\gmax$ is generated by at most $n$ elements. Also, if any of the $a_{i} = 1$, then $\gmax$ is generated by fewer than $n$ elements. In general, this will give us a minimal set of generators for $\gmax$.

We now have an algorithm to compute the generators of $\gmax$. It accepts as input the polynomial exponent matrix $A$, but unlike Algorithm \ref{alg:submatrix-gmax-improved} which outputs the entire group it just outputs a set of column vectors that generate $\gmax$.

\begin{algorithm}[H]
  \caption{Compute the generators of $\gmax$ via the Smith normal form
    \label{alg:smith-form-gmax}}
  \begin{algorithmic}[1]
    \Statex
    \Function{GmaxGens}{$A$}
      \State Compute the Smith normal form $S = P A Q$ \Comment{where $\{a_{i}\}_{i = 1}^{n}$ are the invariant factors}
      \State Scale the $i$th column of $Q$ by $1 / a_{i}$ for each $a_{i} \ne 1$
      \State \Return{the scaled columns of $Q$}
    \EndFunction
  \end{algorithmic}
\end{algorithm}

We will once again see how the algorithm works on an example.

\begin{example}
Consider again the polynomial $W = x^{3} + y^{3} + x^{2}y$, with exponent matrix
\begin{align*}
A = \begin{bmatrix}3 & 0 \\ 0 & 3 \\ 2 & 1 \end{bmatrix}.
\end{align*}
We compute the Smith normal form $S = PAQ$ and get
\begin{align*}
\begin{bmatrix}1 & 0 \\ 0 & 3 \\ 0 & 0 \end{bmatrix} = \begin{bmatrix}0 & 0 & 1 \\1 & 0 & 0 \\ 2 & 1 & -3 \end{bmatrix} \begin{bmatrix}3 & 0 \\ 0 & 3 \\ 2 & 1 \end{bmatrix}\begin{bmatrix}0 & 1 \\ 1 & -2 \end{bmatrix}.
\end{align*}
Scaling the first column of $Q$ by 1 and the second column of $Q$ by $1/3$ (again writing column vectors as row vectors), we get the group generators $(0,1)$ and $\left(\frac{1}{3}, - \frac{2}{3} \right)$. Putting these in canonical form, we see that the first generator is (0,0) which contributes nothing. The second generator becomes $\left(\frac{1}{3}, \frac{1}{3} \right)$, which gives us
\[G_{W}^{\max} = \{(0,0), (1/3, 1/3), (2/3,2/3)\}\]
as we saw in \exampleref{ex-alg-1}.
\end{example}

To analyze this algorithm, we take advantage of the following notation.

\begin{definition}[Soft-Oh Notation]
We write $f(n) \in O^{\sim}(g(n))$ if $f(n) \in O(g(n) \log^{k}g(n))$ for some $k$. 
\end{definition}

We will also need a way to compute the Smith normal form of a matrix. Deterministic polynomial time algorithms exist to compute the Smith normal form of a matrix over $\Z$. One such algorithm is given in the following

\begin{theorem}[\cite{Sto}]\label{stor_alg}
Let $M$ be an $m \times n$ matrix over $\Z$. Suppose two $m \times m$ matrices can be multiplied in $O(m^{\theta})$ steps over $\Z$, and that $B(t)$ bounds the cost of multiplying two $\lceil t \rceil$-bit integers. Then the time required to compute the Smith normal form of $M$ is given by
\begin{align*}
O^{\sim}(m^{\theta - 1}n \cdot B(m \log \|A\|)).
\end{align*}
\end{theorem}

We are now ready to analyze Algorithm \ref{alg:smith-form-gmax}.

\begin{theorem}\label{smith_form_analysis}
Algorithm \ref{alg:smith-form-gmax} runs in polynomial time.
\end{theorem}


\begin{proof}
We see that on line 3 the algorithm performs $n^{2}$ multiplications to appropriately scale the columns of $Q$. The interesting part of the analysis comes from computing the Smith normal form on line 2. Using the method of \cite{VWill}, we can perform matrix multiplication in $O(m^{2.373})$ steps. Using the method by \cite{schonhage1971fast}, we can multiply $n$-bit vectors in $O(n \log n \log \log n)$ steps. We see that Storjohann's algorithm in \thmref{stor_alg} requires 
\[O^{\sim}(m^{2.373}n \log\|A\| \log(m \log\|A\|) \log \log (m \log \|A\|)  )\]
bit operations. Since we always have $m \ge n$, the complexity of Algorithm \ref{alg:smith-form-gmax} reduces to the complexity of computing Storjohann's algorithm. Therefore our algorithm runs in polynomial time when $m$ and $\|A\|$ are polynomial in the size of $n$.

\end{proof}

We now revisit a previous example to compare the running times of Algorithms \ref{alg:submatrix-gmax-improved} and \ref{alg:smith-form-gmax}.

\begin{example}
Consider again the polynomial $W_{n}$ from \lemref{example_poly} and \thmref{submatrix_alg_analysis} given by
\[W_{n} = x_{1}^{2n} + \dots + x_{n}^{2n} + x_{1}^{n}x_{2}^{n} + \dots + x_{n}^{n}x_{1}^{n}.\]
Here $m = 2n$, and $\lfloor \max \{ \frac{1}{q_{i}} \} \rfloor = 2n$. Therefore by \thmref{smith_form_analysis}, Algorithm \ref{alg:smith-form-gmax} runs in $O^{\sim}(n^{3.373})$ steps.
\end{example}

\section{Conclusion}

In this paper we examined two ways to algorithmically compute $\gmax$ in the case where $W$ is a noninvertible polynomial. The natural extension of the algorithm for invertible polynomials proved to be intractable in some cases, while a new algorithm based on the Smith normal form turned out to have a polynomial running time.



\begin{thebibliography}{16}
\providecommand{\natexlab}[1]{#1}
\providecommand{\url}[1]{\texttt{#1}}
\expandafter\ifx\csname urlstyle\endcsname\relax
  \providecommand{\doi}[1]{doi: #1}\else
  \providecommand{\doi}{doi: \begingroup \urlstyle{rm}\Url}\fi

\bibitem[Artebani et~al.(2014)Artebani, Boissi\`{e}re, and Sarti]{ABS11}
M.~Artebani, S.~Boissi\`{e}re, and A.~Sarti.
\newblock The berglund-h{\"u}bsch-chiodo-ruan mirror symmetry for k3 surfaces.
\newblock \emph{Journal de Math\'ematiques Pures et Appliqu\'ees}, 102\penalty0
  (4):\penalty0 758–781, 2014.

\bibitem[Berglund and Henningson(1995)]{Hen}
P.~Berglund and M.~Henningson.
\newblock Landau-ginzburg orbifolds, mirror symmetry and the elliptic genus.
\newblock \emph{Journal de Math\'ematiques Pures et Appliqu\'ees}, 433\penalty0
  (2):\penalty0 311–332, 1995.

\bibitem[Berglund and H{\"u}bsch(1993)]{BH}
P.~Berglund and T.~H{\"u}bsch.
\newblock A generalized construction of mirror manifolds.
\newblock \emph{Nucl. Phys. B}, 393:\penalty0 377–391, 1993.

\bibitem[Cordner(2015)]{Cor15}
N.~Cordner.
\newblock Transposing noninvertible polynomials.
\newblock \emph{Rose-Hulman Undergraduate Mathematics Journal}, 16\penalty0
  (2):\penalty0 54--66, 2015.

\bibitem[Fan et~al.(2013)Fan, Jarvis, and Ruan]{FJR07}
H.~Fan, T.~J. Jarvis, and Y.~Ruan.
\newblock The witten equation, mirror symmetry and quantum singularity theory.
\newblock \emph{Annals of Mathematics}, 178\penalty0 (1):\penalty0 1–106,
  2013.

\bibitem[Francis et~al.(2012)Francis, Jarvis, Johnson, and Suggs]{FJJS11}
A.~Francis, T.~Jarvis, D.~Johnson, and R.~Suggs.
\newblock Landau-ginzburg mirror symmetry for orbifolded frobenius algebras.
\newblock \emph{Proceedings of Symposia in Pure Mathematics}, 85:\penalty0
  333–353, 2012.

\bibitem[Intriligator and Vafa(1990)]{IV}
K.~Intriligator and C.~Vafa.
\newblock Landau-ginzburg orbifolds.
\newblock \emph{Nuclear Phys. B}, 339\penalty0 (1):\penalty0 95–120, 1990.

\bibitem[Kaufmann(2002)]{Ka1}
R.~Kaufmann.
\newblock Orbifold frobenius algebras, cobordisms and monodromies.
\newblock \emph{Cont. Math}, 310:\penalty0 135–161, 2002.

\bibitem[Kaufmann(2003)]{Ka2}
R.~Kaufmann.
\newblock Orbifolding frobenius algebras.
\newblock \emph{Internat. J. Math.}, 14\penalty0 (6):\penalty0 573–617, 2003.

\bibitem[Kaufmann(2006)]{Ka3}
R.~Kaufmann.
\newblock Singularities with symmetries, orbifold frobenius algebras and mirror
  symmetry.
\newblock \emph{Cont. Math.}, 403:\penalty0 67--116, 2006.

\bibitem[Krawitz(2010)]{Kra09}
M.~Krawitz.
\newblock \emph{FJRW rings and Landau-Ginzburg mirror symmetry}.
\newblock PhD thesis, University of Michigan, 2010.

\bibitem[Kreuzer and Skarke(1992)]{KS}
M.~Kreuzer and H.~Skarke.
\newblock On the classification of quasihomogeneous functions.
\newblock \emph{Comm. Math. Phys.}, 150\penalty0 (1):\penalty0 137–147, 1992.

\bibitem[Sch{\"o}nhage and Strassen(1971)]{schonhage1971fast}
A.~Sch{\"o}nhage and V.~Strassen.
\newblock Fast multiplication of large numbers.
\newblock \emph{Computing}, 7\penalty0 (3-4):\penalty0 281--+, 1971.

\bibitem[Storjohann(1996)]{Sto}
A.~Storjohann.
\newblock Near optimal algorithms for computing smith normal forms of integer
  matrices.
\newblock \emph{Proceedings of ISSAC'96 (Z{\"u}rich)}, page 267–274, 1996.

\bibitem[Webb(2013)]{Webb}
R.~Webb.
\newblock Private communication, 2013.

\bibitem[Williams(2012)]{VWill}
V.~V. Williams.
\newblock Multiplying matrices faster than coppersmith-winograd.
\newblock \emph{STOC '12 Proceedings of the forty-fourth annual ACM symposium
  on Theory of computing}, pages 887--898, 2012.

\end{thebibliography}
\end{document}